\documentclass[a4paper,11pt]{amsart}

\usepackage{amsmath,amssymb,amsthm,amsfonts}
\usepackage{color}
\usepackage{hyperref}
\usepackage{mathtools}
\usepackage{bbm}
\hypersetup{colorlinks,linkcolor=[rgb]{0,0,1},citecolor=[rgb]{1,0,0}}
\usepackage{dutchcal}
\usepackage{aliascnt} 
\usepackage{tikz}
\tikzset{>=to}
\usetikzlibrary{cd}
\usetikzlibrary{calc}
\usetikzlibrary{decorations.markings}

\newcommand{\ie}{i.e.\ }

\newcommand{\AIM}[1]{\noindent #1 \smallskip}

\newtheorem{theorem}{Theorem}[section]

\newtheorem*{theorem*}{Theorem}

  \newaliascnt{conjecture}{theorem}
  
  \aliascntresetthe{conjecture}
  
  \newaliascnt{proposition}{theorem}
  \newtheorem{proposition}[proposition]{Proposition}
  \aliascntresetthe{proposition}
  
  \newaliascnt{lemma}{theorem}
  \newtheorem{lemma}[lemma]{Lemma}
  \aliascntresetthe{lemma}
  
  \newaliascnt{corollary}{theorem}
  \newtheorem{corollary}[corollary]{Corollary}
  \aliascntresetthe{corollary}

\theoremstyle{definition}
  \newaliascnt{definition}{theorem}
  \newtheorem{definition}[definition]{Definition}
  \aliascntresetthe{definition}
  
  \newaliascnt{remark}{theorem}
  \newtheorem{remark}[remark]{Remark}
  \aliascntresetthe{remark}
  
  \newaliascnt{question}{theorem}
  
  \aliascntresetthe{question}
  
  \newtheorem*{example*}{Example}
  \newaliascnt{example}{theorem}
  \newtheorem{example}[example]{Example}
  \aliascntresetthe{example}
  
  \newaliascnt{examples}{theorem}
  \newtheorem{examples}[examples]{Examples}
  \aliascntresetthe{examples}
  
\newcommand{\kk}{\mathbbm k}
\newcommand{\ZZ}{\mathbb Z}
\newcommand{\CC}{\mathbb C}
\newcommand{\QQ}{\mathbb Q}
\newcommand{\PP}{\mathbb P}
\newcommand{\cO}{\mathcal O}
\newcommand{\cT}{\mathcal T}
\newcommand{\cN}{\mathcal N}
\newcommand{\cE}{\mathcal E}
\newcommand{\cF}{\mathcal F}
\newcommand{\cG}{\mathcal G}
\newcommand{\cH}{\mathcal H} 
\DeclareMathOperator{\Spec}{Spec}
\DeclareMathOperator{\rk}{rk}

\newcommand{\cA}{\mathcal A}
\newcommand{\cB}{\mathcal B}
\newcommand{\cC}{\mathcal C}
\newcommand{\cD}{\mathcal D}
\newcommand{\Db}{\mathcal D^b}
\DeclareMathOperator{\Aut}{Aut}
\DeclareMathOperator{\Isom}{O}
\newcommand{\cI}{\mathcal I}

\newcommand{\cP}{\mathcal P}
\newcommand{\R}{\mathsf R}
\let\L\relax
\newcommand{\L}{\mathsf L}
\newcommand{\T}{\mathsf T}
\newcommand{\pairing}[1]{\langle #1 \rangle} 
\newcommand{\blank}{-} 
\makeatletter
\newcommand*\cpx{{\mathpalette\cpx@{.666}}}
\newcommand*\cpx@[2]{\mathbin{\vcenter{\hbox{\scalebox{#2}{$\m@th#1\bullet$}}}}}
\makeatother
\newcommand{\ltensor}{\stackrel{\L}{\otimes}}
\newcommand{\Ltensor}{\stackrel{\smash{\L}}{\otimes}}
\DeclareMathOperator{\Cone}{C}
\DeclareMathOperator{\Com}{Com}
\DeclareMathOperator{\Hot}{Hot}
\DeclareMathOperator{\qis}{\mathcal{qis}}
\let\mod\relax
\mathchardef\hh="2D 
\DeclareMathOperator{\mod}{mod}
\DeclareMathOperator{\Mod}{Mod}
\DeclareMathOperator{\hmod}{\hh\mod}
\DeclareMathOperator{\hMod}{\hh\Mod}
\DeclareMathOperator{\coh}{coh}
\DeclareMathOperator{\qcoh}{Qcoh}

\DeclareMathOperator{\im}{im}
\DeclareMathOperator{\coker}{coker}
\DeclareMathOperator{\Hom}{Hom}
\DeclareMathOperator{\SHom}{\mathcal{Hom}}
\DeclareMathOperator{\Ext}{Ext}
\DeclareMathOperator{\Tor}{Tor}
\DeclareMathOperator{\id}{id}
\newcommand{\osemiplus}{ 
\mathbin{%
\raisebox{-0.25ex}{%
  \begin{tikzpicture}
    \draw[line width=0.05ex] (0,0) circle (0.125);
    \draw[line width=0.05ex] (0,0.125) to (0,-0.125);
    \draw[line width=0.05ex] (-0.125,0) to (0,0);
  \end{tikzpicture}}%
}%
}

\newcommand{\xto}{\xrightarrow}

\newcommand{\into}{\hookrightarrow}
\newcommand{\onto}{\twoheadrightarrow}
\newcommand{\isom}{\mathbin{\ooalign{\hss\raisebox{0.7ex}{$\scriptstyle\sim$}\hss\cr$\longrightarrow$}}}
\newcommand{\isomtext}{\mathbin{\ooalign{\hss\raisebox{0.9ex}{$\scriptstyle\sim\,$}\hss\cr$\rightarrow$}}}

\newcommand{\vimende}{\color{white} \\ \color{gray} 
\vfill\noindent\texttt{:wq}}

\title[Introduction to derived categories]{Introduction to derived categories of coherent sheaves}

\author{Andreas Hochenegger}
\address{Dipartimento di Matematica ``Federigo Enriques'', Universit\`a degli Studi di Milano, Italia}
\email{andreas.hochenegger@unimi.it \vimende}

\begin{document}

\maketitle

\begin{abstract}
In these notes, an introduction to derived categories and derived functors is given.
The main focus is the bounded derived category of coherent sheaves on a smooth projective variety.
\end{abstract}

\addtocontents{toc}{\protect\setcounter{tocdepth}{1}} 
\tableofcontents

\section{Introduction}

\noindent
One way to get your hands on coherent sheaves is by short exact sequences.
To name three important ones:
\begin{itemize}
\item 
\makebox[6.5cm][l]{$\displaystyle 0 \to \cO_{\PP^n} \to \cO_{\PP^n}(1)^{\oplus(n+1)} \to \cT_{\PP^n} \to 0$}
\emph{Euler sequence} 
\item 
\makebox[5.5cm][l]{$\displaystyle 0 \to \cI_{Y|X} \to \cO_X \to \cO_Y \to 0$}
\emph{Ideal sheaf sequence}
\item 
\makebox[5.5cm][l]{$\displaystyle 0 \to \cT_Y \to \cT_X|_Y \to \cN_{Y|X} \to 0$}
\emph{Normal sheaf sequence}
\end{itemize}
where $Y \subset X$ is a closed embedding.

Such sequences are usually the starting point for computations.
But by applying any meaningful operation to such a sequence one will almost inevitably lose the exactness on the left or right end.
Examples for such operations are
\\
\begin{minipage}[t]{0.45\linewidth}
\begin{itemize}
\item $\SHom(\cF,\blank)$, $\Hom(\cF,\blank)$
\item $f_*$, $\Gamma(X,\blank)$
\end{itemize}
\end{minipage}
\begin{minipage}[t]{0.45\linewidth}
\begin{itemize}
\item $\cF \otimes \blank$
\item $f^*$
\end{itemize}
\end{minipage}
\\
where $\cF$ is a coherent sheaf, and $f \colon X \to Y$ a morphism.
Another issue is that the projection formula and flat base change work only for specific classes of coherent sheaves such as for locally free sheaves.
That exactness gets lost, should not be seen as a failure but an indication that there is something more to say.

\begin{example*}
Let $C$ and $C'$ be two rational curves on a smooth projective surface $X$.
Applying $\Hom(\blank,\cO_C)$ to the ideal sheaf sequence of $C'$ yields
\[
0 \to \Hom(\cO_{C'},\cO_C) \to H^0(\cO_C) \to H^0(\cO_C(C'))
\]
This sequence is a short exact sequence if and only if
\begin{itemize}
\item $C$ and $C'$ are disjoint, then $\Hom(\cO_{C'},\cO_C) =0$ and $\cO_C \cong \cO_C(C')$; or
\item $C = C'$ and $H^0(\cO_C(C))$ vanishes, as $\Hom(\cO_{C},\cO_C) \to H^0(\cO_C)$ is an isomorphism.
\end{itemize}
These are quite special situations (note that the second case implies that $C^2<0$).
In particular, if $C$ and $C'$ intersect, this sequence can to be continued with $\Ext$-groups. 
The intersection number can be easily computed using the Euler characteristic:
\[
\begin{split}
C'.C & = -\chi(\cO_{C'},\cO_C) = 
-\dim \Hom(\cO_{C'},\cO_C) + \dim \Ext^1(\cO_{C'},\cO_C).
\end{split}
\]
\end{example*}

In order to deal with such examples,
homological algebra proposes to replace sheaves by adapted resolutions and the derived category of sheaves will become the proper framework for such computations. 

\subsubsection*{Aim}

These notes serve as a companion to the lecture notes \cite{Macri-Stellari} and give the necessary background on derived categories.
The motivating question is how to change (or better: derive) a functor between abelian categories in order to keep exactness.
We hope to convince the reader that this question leads quite inevitably to the notion of a derived category and derived functors.

In the first part, we give the general construction of derived categories of an abelian category and derived functors.
The motivating question leads to the notion of adapted resolutions and quasi-isomorphisms in \autoref{sec:adapted}. As an intermediate step we arrive at the notion of a homotopy category in \autoref{sec:homotopy}.
In \autoref{sec:derived}, the derived category is constructed and its triangulated structure discussed.
Finally in \autoref{sec:functors}, we will see that derived functors become exact on the derived level.

In the second part, we focus on the derived category of sheaves, especially on the construction of derived functors.
There we deal with left-exact functors like $\Hom$ and push-forward in \autoref{sec:deriveleft}, and then with right-exact functors like $\otimes$ and pull-back in \autoref{sec:deriveright}.
Moreover, we give some compatibilities among these functors in \autoref{sec:compatibilities}.
Finally, we discuss a bit the important notion of Fourier-Mukai transforms in \autoref{sec:fourier}.

The third section can be seen as an application of the theory of Fourier-Mukai transforms. Moreover, it should pave the way for \cite{Macri-Stellari}. There we present some comparatively recent results on the auto-equivalences of the derived category of a complex projective K3 surface.

For full details, we refer to the wonderful books \cite{Gelfand-Manin} and \cite{Huybrechts} which these notes follow to quite some extent.
But we also want to mention the books \cite{Hartshorne-residues}, \cite{Kashiwara-Schapira} and \cite{Lipman} which were very helpful when compiling these notes.
In this text, we do not give proper references, because all results are nowadays pretty standard and can be found in any of the above mentioned sources. The only exception is the last section were more recent results are presented and therefore some references given.
We want to stress that most of the proofs below are just indications of the main ideas, and usually borrowed from one of the above mentioned books.
We hope that these indications give the novice a good feeling about what is going on, and ideally leave such a reader well-prepared for a closer study using a textbook.

\subsubsection*{Prerequisites}

We assume that the reader has a background in algebraic geometry and is acquainted with basic notions from homological algebra.

\subsubsection*{Conventions}
With $\kk$, we denote a field which is not necessarily algebraically closed or of characteristic zero.
When we speak of categories, we implicitly assume that they are $\kk$-linear (even though this is not strictly necessary for most of the abstract theory).
By a variety we mean an integral separated scheme of finite type over $\kk$.

\subsection*{Acknowledgements}
The author thanks Klaus Altmann, Andreas Krug, Ciaran Meachan, David Ploog and Paolo Stellari for comments and suggestions.

\section{From abelian to derived categories}

\subsection{Adapted resolutions}
\label{sec:adapted}

\AIM{
In this section, we will introduce the central notion of quasi-isomorphism and speak about adapted resolutions. Moreover, we will give a definition of the derived category by a universal property.
}

We fix some notation.
Let $\cA$ be an abelian category, \ie we can speak of short exact sequences.
We denote by $\Com(\cA)$ the \emph{category of (cochain) complexes} in $\cA$, \ie its objects are sequences
\[
C^\cpx 
\colon \quad 
\cdots \to C^{i-1} \xto{d^{i-1}} C^i \xto{d^i} C^{i+1} \to \cdots
\]
with $C^i \in \cA$ and $d^{i} \circ d^{i-1} = 0$ for all $i \in \ZZ$, and morphisms are \emph{maps of complexes}, \ie 
\[
\begin{tikzcd}
C^\cpx \ar[d, "f^\cpx"']
& \cdots \ar[r] & C^i \ar[r, "d^i"] \ar[d, "f^i"'] \ar[dr, phantom, "\displaystyle \circlearrowleft"]
& C^{i+1} \ar[r] \ar[d, "f^{i+1}"] & \cdots\\
D^\cpx 
& \cdots \ar[r] & D^i \ar[r, "d^i"]              & D^{i+1} \ar[r] & \cdots
\end{tikzcd}
\]
With $\Com^+(\cA)$, $\Com^-(A)$ and $\Com^b(\cA)$ we denote the full subcategory of bounded below, bounded above, and bounded complexes, respectively. For example, $C^\cpx \in \Com^+(\cA)$ if $C^i = 0$ for $i \ll 0$.

By slight abuse of notation, given some class of objects $\cI$ in $\cA$,
we will write $\Com(\cI)$ for the (full) subcategory consisting of those complexes in $\Com(\cA)$ which are sequences of objects in $\cI$.

Due to $d^i \circ d^{i-1} = 0$ or, equivalently, $\im d^{i-1} \subseteq \ker d^i$, we can take \emph{cohomology} of any $C^\cpx \in \Com(\cA)$, \ie
\[
\cH^i(C^\cpx) = \frac{\ker d^i}{\im d^{i-1}}.
\]
Note that we can consider $\cH^\cpx(C^\cpx)$ when equipped with the zero-differential again as an element of $\Com(\cA)$.
We say that $C^\cpx$ is \emph{acyclic} (or \emph{exact}) if it has no cohomology, \ie $\cH^\cpx(C^\cpx) = 0$.

Moreover, a map $f^\cpx \colon C^\cpx \to D^\cpx$ of complexes induces a map
\[
\cH^\cpx(f^\cpx) \colon \cH^\cpx(C^\cpx) \to \cH^\cpx(D^\cpx).
\]
We say that $f^\cpx$ is a \emph{quasi-isomorphism} if $\cH^\cpx(f^\cpx)$ is an isomorphism.

\begin{definition}
Let $F \colon \cA \to \cB$ be a left-exact functor between abelian categories.
Let $\cI_F$ be a class of objects in $\cA$.
We say that $\cI_F$ is \emph{$F$-adapted} if
\begin{itemize}
\item $F(I^\cpx)$ is acyclic for any acyclic complex $I^\cpx \in \Com^+(\cI_F)$;
\item for any $A \in \cA$ there is an injection $A \into I$ with $I \in \cI_F$.
\end{itemize}
\end{definition}

The first property says in particular that $F$ preserves exactness of short exact sequences of objects in $\cI_F$.
The second property ensures that we can replace any $A$ by an adapted resolution, as the following lemma shows.

\begin{lemma}
\label{lem:injective-resolution}
Let $F \colon \cA \to \cB$ be a left-exact functor between abelian categories,
and let $\cI_F$ be an $F$-adapted class.
Then for any $A \in \cA$, there is a complex $I^\cpx \in \Com^+(\cI_F)$ such that
\[
\begin{tikzcd}
A \ar[d, "f"'] 
& \cdots \ar[r] & 0 \ar[r] & A \ar[r] \ar[d, "f"] & 0 \ar[r] & \cdots \\
I^\cpx
& \cdots \ar[r] & 0 \ar[r] & I^0 \ar[r] & I^1 \ar[r] & \cdots
\end{tikzcd}
\]
is a quasi-isomorphism.
We call $I^\cpx$ an \emph{$F$-adapted resolution} of $A$.
\end{lemma}

\begin{proof}
We only indicate how $I^\cpx$ can be constructed.
By the second property of an $F$-adapted class, there is an injection
$f \colon A \into I^0$ for some $I^0 \in \cI_F$.
Now continue inductively, by choosing $I^{i+1}$ to contain the cokernel of the previous map, and setting $d^i \colon I^i \to I^{i+1}$ to be the composition $I^i \onto \coker \into I^{i+1}$.
\end{proof}

\begin{remark}
\label{rem:replace-complex}
The above lemma can be generalised to complexes, \ie
for any $A^\cpx \in \Com^+(\cA)$ there is an adapted $I^\cpx \in \Com^+(\cI_F)$ and a quasi-isomorphism $f^\cpx \colon A^\cpx \to I^\cpx$.
\end{remark}

\begin{proposition}
If $\cA$ contains enough injective objects, \ie for any $A \in \cA$ there is an inclusion $A \into I$ with $I$ injective,
then the class $\cI_{\cA}$ of all injective objects in $\cA$ is adapted for all left-exact functors starting in $\cA$.
\end{proposition}

\begin{proof}
This question can be reduced to short exact sequences, by breaking up $I^\cpx$ into $0 \to \ker(d^i) \to I^i \to \im(d^i) \to 0$.
Now the statement can be shown using two standard facts about injective objects:
\begin{itemize}
\item Any short exact sequence $0 \to I \to A \to B \to 0$ in $\cA$ with $I$ injective splits. In particular, its image under $F$ is still exact.
\item For a short exact sequence $0 \to I' \to I \to A \to 0$ in $\cA$ with $I,I'$ injective, also $A$ is injective. \qedhere
\end{itemize}
\end{proof}

\begin{remark}
We have dealt here only with left-exact functors, but there is a dual story.
For a right-exact functor $F \colon \cA \to \cB$, an $F$-adapted class $\cP_F$ should satisfy
\begin{itemize}
\item $F(P^\cpx)$ is acyclic for any acyclic complex $P^\cpx \in \Com^-(\cP_F)$;
\item for any $A \in \cA$ there is an surjection $P \onto A$ with $P \in \cP_F$.
\end{itemize}
Moreover, we get an $F$-adapted resolution $P^\cpx \to A$ in $\Com^-(\cP_F)$.
Finally, if there are enough projective objects, the class of projective objects is adapted for all right-exact functors.
\end{remark}

The discussion of this section shows, that we want to identify quasi-isomorphic complexes, as such an identification allows us to pass from an object to an adapted resolution.
This aim is summarised in the following definition.

\begin{definition}
Let $\cA$ be an abelian category.
A category $\cD$ together with a functor $Q \colon \Com(\cA) \to \cD$ is called \emph{derived category} of $\cA$ if
\begin{itemize}
\item $Q(f^\cpx)$ is an isomorphism for any quasi-isomorphism $f^\cpx$;
\item any other functor $F \colon \Com(\cA) \to \cT$ which maps quasi-isomorphisms to isomorphism factors uniquely through $\cD$:
\[
\begin{tikzcd}[column sep=small]
\mathllap{\Com}(\cA) \ar[rr, "F"] \ar[rd, "Q"'] && \cT\\
& \cD \ar[ru, dashrightarrow, "\exists!"']
\end{tikzcd}
\]
\end{itemize}
Analogously, we can define the bounded below, bounded above and bounded derived category of $\cA$.
\end{definition}

This definition by a universal property automatically yields the uniqueness up to equivalence, but we have yet to provide existence.

\subsection{The homotopy category}
\label{sec:homotopy}

\AIM{
In this section, we will introduce homotopies and show that they induce quasi-isomorphisms. In the case that there are enough injective objects (or dually, projective objects), these are all quasi-isomorphisms.
}

There is a cheap way to build a map of complexes:

\begin{lemma}
Let $C^\cpx$ and $D^\cpx$ be two complexes and $\{h^i \colon C^i \to D^{i-1}\}_i$ be a sequence of morphisms in $\cA$.
Then $f^i = h^{i+1} d^i + d^{i-1} h^i \colon C^i \to D^i$ fit together to a map of complexes:
\[
\begin{tikzcd}
C^\cpx \ar[d, "f^\cpx"'] & & \cdots \ar[r] & C^i \ar[r, "d^i"] \ar[dl, "h^i"'] \ar[d, "f^i"] & C^{i+1} \ar[r] \ar[dl, "h^{i+1}"] & \cdots \\
D^\cpx & \cdots \ar[r] & D^{i-1} \ar[r, "d^{i-1}"'] & D^i \ar[r] & \cdots
\end{tikzcd}
\]
\end{lemma}

\begin{proof}
We only have to check that
$d^i (h^{i+1} d^i + d^{i-1} h^i) = (h^{i+2} d^{i+1} + d^i h^{i+1}) d^i$, which holds as $C^\cpx$ and $D^\cpx$ are complexes.
\end{proof}

\begin{definition}
Let $f^\cpx, g^\cpx \colon C^\cpx \to D^\cpx$ be two maps of complexes.
We say that $f^\cpx$ and $g^\cpx$ are \emph{homotopic}, 
if there is a sequence of morphisms $h^i \colon C^i \to D^{i-1}$,
such that $f^i - g^i = h^{i+1} d^i + d^{i-1} h^i$ for all $i \in \ZZ$.
We write $f^\cpx \sim g^\cpx$ in this case.
\end{definition}

\begin{lemma}
Let $f^\cpx, g^\cpx \colon C^\cpx \to D^\cpx$ be two maps of complexes.
If $f^\cpx$ and $g^\cpx$ are homotopic, then the induced maps $\cH^\cpx(f^\cpx)$ and $\cH^\cpx(g^\cpx)$ are equal.
Moreover, homotopy $\sim$ defines an equivalence relation for maps of complexes.
\end{lemma}

As a corollary we get that homotopies are a source of quasi-isomorphisms:

\begin{remark}
\label{rem:homotopies-quasiisos}
Let $f^\cpx \colon C^\cpx \to D^\cpx$ and $g^\cpx \colon D^\cpx \to C^\cpx$ be two maps of complexes such that $f^\cpx \circ g^\cpx \sim \id_{D^\cpx}$ and $g^\cpx \circ f^\cpx \sim \id_{C^\cpx}$.
Then both $f$ and $g$ are quasi-isomorphisms, as
\[
\cH^\cpx(f^\cpx) \circ \cH^\cpx(g^\cpx) = \cH^\cpx(f^\cpx \circ g^\cpx) = \cH^\cpx(\id_{D^\cpx}) = \id_{\cH^\cpx(D^\cpx)}
\]
and similarly for $\cH^\cpx(g^\cpx) \circ \cH^\cpx(f^\cpx)$.
\end{remark}

\begin{definition}
Let $\cA$ be an abelian category.
The \emph{homotopy category} $\Hot(\cA)$ of $\cA$ consists of
\begin{itemize}
\item objects: complexes of objects in $\cA$;
\item morphisms: maps of complexes modulo homotopy
\[
\Hom_{\Hot(\cA)}(C^\cpx,D^\cpx) \coloneqq \Hom_{\Com(\cA)}(C^\cpx,D^\cpx) / \sim
\]
\end{itemize}
Moreover, we can define $\Hot^+(\cA)$, $\Hot^-(\cA)$ and $\Hot^b(\cA)$ as the full subcategories of $\Hot(\cA)$ consisting
of bounded below, bounded above, and bounded complexes, respectively.
Similarly for any full additive subcategory $\cC$ of $\cA$, we can define the homotopy category $\Hot(\cC)$ (and bounded analogues) by restricting to complexes of objects in $\cC$.
\end{definition}

\subsection*{Enough injective objects}
In the case that enough injectives are present, we can say even more about the homotopy category.

\begin{proposition}
\label{prop:lifting}
Let $f\colon C \to D$ be a morphism in an abelian category $\cA$ with enough injectives.
Then the following holds
\begin{itemize}
\item For any choice of injective resolutions $C \to I^\cpx$ and $D \to J^\cpx$,
$f$ can be lifted to a map of complexes $f^\cpx \colon I^\cpx \to J^\cpx$, in particular
the following diagram commutes:
\[
\begin{tikzcd}
C \ar[r, hook] \ar[d, "f"'] & I^0 \ar[d, "f^0"] \\
D \ar[r, hook] & J^0 
\end{tikzcd}
\]
\item any two such lifts are homotopic.
\end{itemize}
\end{proposition}

\begin{proof}
We only show existence, because uniqueness up to homotopy can be shown similarly.
By injectivity of $J^0$, there is the lift $f^0$ of the composition $C \to D \into J^0$:
\[
\begin{tikzcd}
C \ar[r, hook] \ar[d, "f"'] & I^0 \ar[d, dashrightarrow, "f^0"] \\
D \ar[r, hook] & J^0 
\end{tikzcd}
\]
The statement can be shown by induction, continuing the argument like in that proof of \autoref{lem:injective-resolution}.
\end{proof}

\begin{remark}
\label{rem:lifting-complex}
Actually, a similar proof which is (notationally) more involved shows that any $f^\cpx \colon C^\cpx \to D^\cpx$ in $\Com^+(\cA)$ for an abelian category $\cA$ with enough injectives can be lifted to a map of complexes $\tilde f^\cpx \colon I^\cpx \to J^\cpx$ with $I^\cpx$ and $J^\cpx$ injective resolutions of $C^\cpx$ and $D^\cpx$.
Again, any two such lifts are homotopic.
\end{remark}

\begin{remark}
For $f^\cpx = \id_{C^\cpx} \colon C^\cpx \to C^\cpx$, we get as an important special case that any two injective resolutions of $C^\cpx$ are homotopic.
\end{remark}

Finally, there is a converse to \autoref{rem:homotopies-quasiisos},
whose proof needs \autoref{rem:lifting-complex}.

\begin{proposition}
\label{prop:injectives-qis}
Let $\cA$ be an abelian category with enough injectives.
Let $f^\cpx \colon I^\cpx \to J^\cpx$ be a quasi-isomorphism of injective complexes in $\Com^+(\cA)$.
Then there is a quasi-isomorphism $g^\cpx \colon J^\cpx \to I^\cpx$ with homotopies $f^\cpx \circ g^\cpx \sim \id_{J^\cpx}$ and $g^\cpx \circ f^\cpx \sim \id_{I^\cpx}$.
\end{proposition}

Given an abelian category $\cA$ with enough injectives $\cI$,
the last proposition shows that quasi-iso\-morphisms become invertible in $\Hot^+(\cI)$, but even more is true.

\begin{proposition}
\label{prop:injectives-hot}
Let $\cA$ be an abelian category with enough injectives $\cI$.
Then $\Hot^+(\cI)$ is the bounded below derived category $\cD^+(\cA)$ of $\cA$.
\end{proposition}

For an arbitrary $F$-adapted class, a quasi-isomorphism might not have a homotopy inverse like in \autoref{prop:injectives-qis}. 
The crucial ingredient there is the lifting property of injective objects.
As usual, we can enforce the existence of such homotopy inverses by formally introducing them.
This will be done in the following section.

\begin{remark}
In the presence of enough projective objects $\cP$ in an abelian category $\cA$, we get statements dual to those in this subsection. 
Most notably, in this case $\Hot^-(\cP)$ is the bounded above derived category $\cD^-(\cA)$ of $\cA$.
\end{remark}

\subsection{The derived category}
\label{sec:derived}

\AIM{
In this section, we will finally give a construction of the derived category and speak about its triangulated structure.
}

For an abelian category $\cA$ let $\qis$ denote the class of all quasi-isomorphisms.
We finally state the existence of the derived category in general, which is due to Verdier.

\begin{theorem}
Let $\cA$ be an abelian category.
The category $\cD(\cA) \coloneqq \Hot(\cA)[\qis^{-1}]$ given by
\begin{itemize}
\item objects: complexes of objects in $\cA$;
\item morphisms: the same as in $\Hot(\cA)$ but with quasi-isomorphisms formally inverted:
\[
\begin{split}
&\Hom_{\cD(\cA)}(C^\cpx,D^\cpx) 
 \coloneqq \Hom_{\Hot(\cA)}(C^\cpx,D^\cpx)[\qis^{-1}]
= \\
& =
\left\{ 
\begin{tikzcd}[column sep=small, ampersand replacement=\&]
\& \tilde C^\cpx \ar[rd, "\tilde f^\cpx"] \ar[ld, "s^\cpx"'] \\
C^\cpx \ar[rr, dashrightarrow, "f"'] 
\&\& D^\cpx
\end{tikzcd}
\,\middle|\,
\begin{array}{l}
\text{ $\tilde C^\cpx$ complex of objects in $\cA$, } \\
\text{ $s^\cpx \in \Hom_{\Hot(\cA)}(C^\cpx,\tilde C^\cpx)$ quasi-isomorphism,}\\
\text{ $\tilde f^\cpx \in \Hom_{\Hot(\cA)}(\tilde C^\cpx, D^\cpx)$. }
\end{array}
\right\}
\end{split}
\]
\end{itemize}
is the derived category of $\cA$.

\end{theorem}

\begin{remark}
The definition of the morphisms above as \emph{roofs} is a bit informal. 
For example one needs to show that a zig-zag of two such roofs can be composed to a single roof. 
For this the key ingredient is that quasi-isomorphisms form a \emph{localising class} of morphisms inside $\Hot(\cA)$.
To invert such a class is also called \emph{Verdier localisation}.

One can construct the derived category of $\cA$ also by formally inverting quasi-isomorphisms in $\Com(\cA)$, see \cite[\S III.2.2]{Gelfand-Manin}.
But this causes several technical problems which can be avoided by passing first to $\Hot(\cA)$.
\end{remark}

\begin{remark}
\label{rem:embedding}
There is a natural functor
\[
\cA \to \cD(\cA),\ C \mapsto [\cdots \to 0 \to C \to 0 \to \cdots]
\]
mapping any $C \in \cA$ to the complex with $C$ at the zero position. By slight abuse of notation, we will denote this complex again by $C$.

This functor is fully faithful, \ie for any two $C,D \in \cA$ holds
\[
\Hom_{\cD(\cA)}(C,D) = \Hom_\cA(C,D).
\]
Moreover, the essential image of this functor consists of all complexes $C^\cpx$ such that $\cH^i(C^\cpx)=0$ for $i \neq 0$.
\end{remark}

\subsection*{Triangulated structure}

As the objects of $\cD(\cA)$ are complexes, there is the \emph{shift functor}:
\[
[1] \colon \cD(\cA) \to \cD(\cA),\ C^\cpx \mapsto C^\cpx[1] \coloneqq C^{\cpx+1}.
\]
The usual convention here is that the sign of the differential changes under shift, \ie 
\[
d^i_{C^\cpx[1]} = - d^{i+1}_{C^\cpx}.
\]
With this shift functor, we can define the \emph{(mapping) cone} of a map of complexes $f^\cpx \colon C^\cpx \to D^\cpx$ as the complex $\Cone(f^\cpx)$ with
\[
\Cone^i(f^\cpx) = C^{i+1} \oplus D^i,\ 
d^i_{\Cone(f^\cpx)} = 
\begin{pmatrix}
-d^{i+1}_{C^\cpx} & 0 \\
f^{i+1} & d^i_{D^\cpx}
\end{pmatrix}
\]
We will also write $\Cone(f^\cpx) = C^\cpx[1] \osemiplus D^\cpx$ as a semi-direct sum.
With these definitions $f$ induces a triangle of morphisms in $\cD(\cA)$:
\[
C^\cpx \xto{f^\cpx} D^\cpx \xto{j^\cpx} \Cone(f^\cpx) \xto{p^\cpx} C^\cpx[1].
\]
where $j^\cpx$ is the inclusion of the semi-direct summand $D^\cpx$ and $p^\cpx$ the projection onto $C^\cpx[1]$.

\begin{remark}
We want to stress that only for honest maps of complexes we have an explicit construction of the mapping cone.
Notationally, we will therefore mark a map of complexes $f^\cpx$ always with a dot, in order to distinguish them from (general) morphisms $f$ in $\cD(\cA)$ which are roofs.
\end{remark}

\begin{definition}
We call a sequence of morphisms $C^\bullet \to D^\bullet \to E^\bullet \to C^\cpx[1]$ an \emph{exact triangle} (or \emph{distinguished triangle})

We call a sequence of morphisms $C^\cpx \to D^\cpx \to E^\cpx \to C^\cpx[1]$ an \emph{exact triangle} (or \emph{distinguished triangle})
if there is a commutative diagram in $\cD(\cA)$ of the form
\[
\begin{tikzcd}
C^\cpx \ar[r, "f"] \ar[d, "c"', "\wr"]  & D^\cpx \ar[r] \ar[d, "d"', "\wr"] & E^\cpx \ar[r] \ar[d, "e"', "\wr"] & C^\cpx[1] \ar[d, "{c[1]}"', "\wr"] \\
C'{}^\cpx \ar[r, "\tilde f^\cpx"] & D'{}^\cpx \ar[r]            & \Cone(\tilde f^\cpx) \ar[r]      & C'{}^\cpx[1]           
\end{tikzcd}
\]
with $\tilde f^\cpx$ a map of complexes.

The complex $E^\cpx$ is called the \emph{cone} of the morphism $f \colon C^\cpx \to D^\cpx$ and denoted by $\Cone(f)$.
\end{definition}

\begin{remark}
The triangle is often visualised in the following way:
\[
\begin{tikzcd}[column sep=small]
C^\cpx \ar[rr] && D^\cpx \ar[ld] \\
& E^\cpx \ar[lu, "{[1]}"]
\end{tikzcd}
\]
where the lower left arrow involves a shift by one.
Note that a triangle can also be extended to a long sequence
\[
\cdots \to E^\cpx[-1] \to C^\cpx \to D^\cpx \to E^\cpx \to C^\cpx[1] \to D^\cpx[1] \to \cdots
\]
which is actually a complex in $\cD(\cA)$, see \autoref{rem:exact} below.
\end{remark}

Distinguished triangles generalise short exact sequences in a very precise way.

\begin{proposition}
Let $0 \to C \xto{f} D \xto{g} E \to 0$ be a short exact sequence in the abelian category $\cA$.
Considering these objects in $\cD(\cA)$, they form the exact triangle
 $C \xto{f} D \xto{g} E \to C[1]$.
\end{proposition}

\begin{proof}
We consider the exact triangle
$C \xto{f} D \xto{j} \Cone(f) \xto{p} C[1]$.
Note that $\Cone(f)$ is a two-term complex, quasi-isomorphic to $E$:
\[
\begin{tikzcd}
\Cone(f) \ar[d, "g"'] & 0 \ar[r] & C \ar[r, "f"] & D \ar[r] \ar[d, "g"] & 0 \\
E                     &          & 0 \ar[r]      & E \ar[r]             & 0
\end{tikzcd}
\]
One can check that this quasi-isomorphism can be completed to a diagram of quasi-isomorphisms, which shows the claim:
\[
\begin{tikzcd}
C \ar[r, "f"] \ar[d, equal] & D \ar[r, "j"] \ar[d, equal] & \Cone(f) \ar[r, "p"] \ar[d, "g \in \qis"] & C[1] \ar[d, equal]\\
C \ar[r, "f"] & D \ar[r, "g"] & E \ar[r, "h"] & C[1] & \qedhere
\end{tikzcd}
\]
\end{proof}

\begin{remark}
The mindful reader may ask about the third morphism in the triangle, namely $h \colon E \to C[1]$.
Note that $h \in \Hom_{\cD(\cA)}(E,C[1]) = \Ext^1_{\cA}(E,C)$, see \autoref{ex:ext}.
It is well-known that $\Ext^1(E,C)$ corresponds to extensions, so $h$ encodes the middle term $D$; see \cite[\S III.6.2]{Gelfand-Manin} for a discussion of this.
\end{remark}

\begin{theorem}
Let $\cA$ be an abelian category.
Then its derived category $\cD(\cA)$ is a \emph{triangulated category}, \ie it satisfies the four axioms {\bf TR1} -- {\bf TR4}.
\end{theorem}

\subsection*{TR1}
The triangle $C^\cpx \xto{\id} C^\cpx \to 0 \to C^\cpx[1]$ is exact.\\
Any triangle isomorphic to an exact one is again exact.\\
Any morphism $f \colon C^\cpx \to D^\cpx$ can be completed to an exact triangle.

\begin{proof}[Proof of {\bf TR1} for $\cD(\cA)$]
For the derived category $\cD(\cA)$ the second clause is satisfied by definition. For the first clause, one only needs to check that the cone $\Cone(\id)$ is homotopic to the zero complex.
To see the last, write a morphism $f \colon C^\cpx \to D^\cpx$ as a roof $f = \tilde f^\cpx \circ (s^\cpx)^{-1}$, which fits into the following commutative diagram of exact triangles:
\[
\begin{tikzcd}
C^\cpx \ar[r, "f"] \ar[d, "(s^\cpx)^{-1}"', "\wr"] & D^\cpx \ar[d, equal] \ar[r, dashrightarrow, "j^\cpx"] & C(f) \ar[r, dashrightarrow, "{s^\cpx[1] \circ p^\cpx}"] \ar[d, equal] & C^\cpx[1] \ar[d, "{(s^\cpx)^{-1}[1]}", "\wr"']\\
\tilde C^\cpx \ar[r, "\tilde f^\cpx"] & D^\cpx \ar[r, "j^\cpx"] & \Cone(f^\cpx) \ar[r, "p^\cpx"] & \tilde C^\cpx[1]
&\qedhere
\end{tikzcd}
\]
\end{proof}

\begin{remark}
\label{rem:conekercoker}
By the last clause of {\bf TR1}, cones in the derived category $\cD(\cA)$ unifying both kernel and cokernel of the abelian category $\cA$.
More precisely, considering a map $f \colon C \to D$ in $\cA$ as a map of complexes in $\cD(\cA)$, one can check that $\cH^{-1}(\Cone(f)) = \ker(f)$ and $\cH^0(\Cone(f)) = \coker(f)$.
\end{remark}

\subsection*{TR2}
The triangle $C^\cpx \xto{f} D^\cpx \xto{g} E^\cpx \xto{h} C^\cpx[1]$ is exact if and only if
$D^\cpx \xto{g} E^\cpx \xto{h} C^\cpx[1] \xto{-f[1]} D^\cpx[1]$ is.

\begin{proof}[Proof of {\bf TR2} for $\cD(\cA)$]
We only discuss ``$\implies$'' a bit (as the converse direction is analoguous).
By {\bf TR1} we may assume that $f = f^\cpx$ is a map of complexes,  $E^\cpx \cong \Cone(f^\cpx)$ and $g^\cpx \colon D^\cpx \to \Cone(f^\cpx)$ is the inclusion as a semi-direct summand.
We have to show that $C^\cpx[1]$ is isomorphic to $\Cone(g^\cpx)$.
Note that by our simplifications $\Cone(g^\cpx) = D^\cpx[1] \osemiplus C^\cpx[1] \osemiplus D^\cpx$. One can now check that
\[
(-f^\cpx[1], \id, 0) \colon D^\cpx[1] \osemiplus C^\cpx[1] \osemiplus D^\cpx \to \Cone(g^\cpx)
\]
gives the desired isomorphism.
\end{proof}

\subsection*{TR3}
Given two exact triangles and two morphisms $c$ and $d$ as below:
\[
\begin{tikzcd}
C^\cpx \ar[r, "f"] \ar[d, "c"] & D^\cpx \ar[r] \ar[d, "d"] & E^\cpx \ar[r] \ar[d, dashrightarrow, "e"] & C^\cpx[1] \ar[d, "{c[1]}"]\\
C'{}^\cpx \ar[r, "f'"] & D'{}^\cpx \ar[r] & E'{}^\cpx \ar[r] & C'{}^\cpx[1]
\end{tikzcd}
\]
then there is a (not necessarily unique) morphism $e$ making this diagram commutative.

\begin{proof}[Proof of {\bf TR3} for $\cD(\cA)$]
After replacing $f$ and $f'$ by maps of complexes, and $E^\cpx$ and $E'{}^\cpx$ by the respective cones, one can check that $e = (c[1],d)$ fits into the diagram.
\end{proof}

\begin{remark}
One might suppose by our reasoning about the existence of the dashed morphism in $\cD(\cA)$, that cones are functorial in $\cD(\cA)$, \ie given a natural transformation $\eta \colon F \to G$ between functors preserving exact triangles, there exist the functor $\Cone(\eta)$ of cones.

In a naive way, such a statement is wrong. Take for example the exact triangle $C^\cpx \xto{\id} C^\cpx \to 0 \to C^\cpx[1]$ for any non-trivial $C^\cpx \in \cD(\cA)$.
After shifting, we can write down the following diagram
\[
\begin{tikzcd}
C^\cpx \ar[r] \ar[d] & 0 \ar[r] \ar[d] & C^\cpx[1] \ar[r, "{-\id[1]}"] \ar[d, dashrightarrow] & C^\cpx[1] \ar[d] \\
0 \ar[r] & C^\cpx[1] \ar[r, "{-\id[1]}"] & C^\cpx[1] \ar[r] & 0
\end{tikzcd}
\]
All the non-labelled solid arrows are just zero morphisms. For the dashed arrow, we can choose any morphism $C^\cpx[1] \to C^\cpx[1]$.

But in a more sophisticated way, such a statement is true for derived categories using dg-enhancements, a topic that we will not enter here.
\end{remark}

\begin{remark}
By \cite[Lem. 2.2]{May}, {\bf TR3} is not necessary as an axiom, it follows from the other three axioms. But we prefer to keep it in this list, as it is an often used property of triangulated categories.
Finally, this shows also that the non-functoriality of cones inside a triangulated category goes deeper than {\bf TR3}.
\end{remark}

\begin{remark}
\label{rem:exact}
From {\bf TR1} -- {\bf TR3} follows that in exact triangles, the composition of two consecutive morphisms is zero.
This follows from the following diagram (and shifted versions):
\[
\begin{tikzcd}
C^\cpx \ar[r, "\id"] \ar[d, "\id"] & C^\cpx \ar[r] \ar[d, "f"] \ar[dr, phantom, "\displaystyle \circlearrowleft"]
 & 0 \ar[r] \ar[d, dashrightarrow] & C^\cpx[1] \ar[d, "\id"] \\
C^\cpx \ar[r, "f"] & D^\cpx \ar[r, "g"] & E^\cpx \ar[r, "h"] & C^\cpx[1]
\end{tikzcd}
\]
\end{remark}

%

\subsection*{TR4}
Given two morphisms $f \colon C^\cpx \to D^\cpx$ and $g \colon D^\cpx \to E^\cpx$, there is a triangle of cones
\[
\Cone(f) \to \Cone(g \circ f) \to \Cone(g) \to \Cone(f)[1]
\]
which fits into the following commutative diagram (where we suppress for simplicity the last degree-increasing morphism in the exact triangles):
\begin{center}
\begin{tikzpicture}
\tikzset{ar/.style={decoration={markings,mark=at position 1 with {\arrow[scale=2]{>}}}, postaction={decorate}}}
\node (A) at (150:1.2cm) {$C^\cpx$};
\node (B) at (30:1.2cm) {$D^\cpx$};
\node (C) at (270:1.2cm) {$E^\cpx$};
\draw [ar] (A) -- (B) node[auto,midway] {$f$};
\draw [ar] (A) -- (C) node[auto,swap,midway] {$g\circ f$};
\draw [ar] (B) -- (C) node[auto,midway] {$g$};
\node (cf) at ($(B) +(2.4cm,0)$) {$\Cone(f)$};
\node (cg) at ($(C) +(-120:2.4cm)$) {$\Cone(g)$};
\node (cgf) at ($(C) +(-60:2.4cm)$) {$\ \ \ \ \Cone(g\circ f)$};
\draw  [ar] (B) -- (cf);
\draw  [ar] (C) -- (cgf);
\draw  [ar] (C) -- (cg);
\draw [dashed,ar] (5:3.6cm) arc (5:-50:3.6cm);
\draw [dashed,ar] (-80:3.6cm) arc (-80:-100:3.6cm);
\end{tikzpicture}
\end{center}

We omit the proof of {\bf TR4} for $\cD(\cA)$, as it is more technical.

\begin{remark}
The last axiom goes under the name \emph{octahedral axiom} as it can be pictured by a diagram in the form of an octahedron, but we think that the above diagram is more helpful.
It comes from the following lemma about abelian categories,
which one might call \emph{windmill lemma}:

Given $f\colon C \to D$ and $g \colon D \to E$ in an abelian category $\cA$.
Then there is an exact sequence of kernels and cokernels fitting into the commutative diagram of \autoref{fig:windmill}.
\begin{figure}
\scalebox{0.8}{
\begin{tikzpicture}
\tikzset{ar/.style={decoration={markings,mark=at position 1 with {\arrow[scale=2]{>}}}, postaction={decorate}}}
\node (A) at (150:1.5cm) {$C$};
\node (B) at (30:1.5cm) {$D$};
\node (C) at (270:1.5cm) {$E$};
\draw [ar] (A) -- (B) node[auto,midway] {$f$};
\draw [ar] (A) -- (C) node[auto,swap,midway] {$g\circ f$};
\draw [ar] (B) -- (C) node[auto,midway] {$g$};
\node (kf) at ($(A) +(-3cm,0)$) {$\ker(f)$};
\node (kgf) at ($(A) +(120:3cm)$) {$\ker(g\circ f)$};
\node (cf) at ($(B) +(3cm,0)$) {$\coker(f)$};
\node (kg) at ($(B) +(60:3cm)$) {$\ker(g)$};
\node (cg) at ($(C) +(-120:3cm)$) {$\coker(g)$};
\node (cgf) at ($(C) +(-60:3cm)$) {$\coker(g\circ f)$};
\draw  [ar] (kf) -- (A);
\draw  [ar] (kgf) -- (A);
\draw  [ar] (kg) -- (B);
\draw  [ar] (B) -- (cf);
\draw  [ar] (C) -- (cgf);
\draw  [ar] (C) -- (cg);
\node (z1) at (195:4.5cm) {$0$};
\node (z2) at (220:4.5cm) {$0$};
\draw [dashed,ar] (165:4.5cm) arc (165:140:4.5cm);
\draw [dashed,ar] (120:4.5cm) arc (120:60:4.5cm);
\draw [dashed,ar] (40:4.5cm) arc (40:15:4.5cm);
\draw [dashed,ar] (5:4.5cm) arc (5:-55:4.5cm);
\draw [dashed,ar] (-80:4.5cm) arc (-80:-100:4.5cm);
\draw [dashed,ar] (-120:4.5cm) arc (-120:-135:4.5cm);
\draw [dashed,ar] (190:4.5cm) arc (190:175:4.5cm);
\end{tikzpicture}
}
\caption{The windmill lemma.}
\label{fig:windmill}
\end{figure}
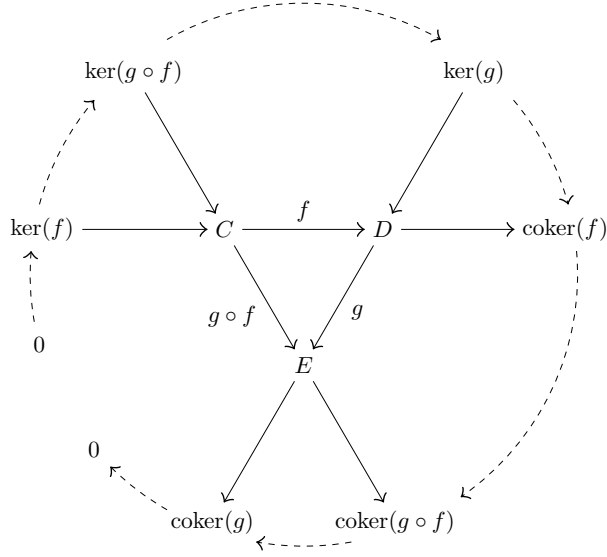
The proof is an exercise in homological algebra,
but may also be deduced from the octahedral axiom using \autoref{rem:conekercoker}.
\end{remark}

\begin{remark}
Just by restricting the class of exact triangles, 
one can also see that $\cD^-(\cA)$, $\cD^+(\cA)$ and $\Db(\cA)$ are triangulated categories.
\end{remark}

\subsection{Exact functors}
\label{sec:functors}

\AIM{
In this section, we introduce the notion of an exact functor between triangulated categories.
}

\begin{definition}
A functor $F \colon \cT \to \cT'$ between triangulated categories is called \emph{exact}, if 
\begin{itemize}
\item $F$ commutes with shifts, \ie there is a functor isomorphism $F \circ [1]_{\cT} \cong [1]_{\cT'} \circ F$;
\item for any exact triangle $A \to B \to C \to A[1]$ in $\cT$, its image $F(A) \to F(B) \to F(C) \to F(A)[1]$ is exact in $\cT'$.
\end{itemize}
\end{definition}


\begin{proposition}
Let $\cA$ be an abelian category and $\cD(\cA)$ its derived category.
Then there is the functor $\cH^\cpx \colon \cD(\cA) \to \cD(\cA)$ which sends each complex $C^\cpx$ to its cohomology $\cH^\cpx(C^\cpx)$ equipped with the zero differential.
This functor is an exact functor. 
\end{proposition}

\begin{remark}
\label{rem:longexact}
The statement of this proposition is usually formulated differently. For any exact triangle $C^\cpx \xto{c} D^\cpx \xto{d} E^\cpx \xto{e} C^\cpx[1]$ in $\cD(\cA)$, its image under $\cH^\cpx$ can be rolled out to a \emph{long exact sequence in cohomology}:
\[
\cdots \to \cH^i(C^\cpx) \xto{\cH^i(c)} \cH^i(D^\cpx) \xto{\cH^i(d)} \cH^i(E^\cpx) \xto{\cH^i(e)} \cH^{i+1}(C^\cpx) \to \cdots
\]
\end{remark}

For defining derived functors in general, we first give an analogue of \autoref{prop:injectives-hot}.

\begin{proposition}
\label{thm:derived-functor}
Let $F \colon \cA \to \cB$ be a left-exact functor, and let $\cI_F$ be an $F$-adapted class.
Then the inclusion $\Com^+(\cI_F) \subset \Com^+(\cA)$ induces an equivalence 
\[
\iota_F \colon \Hot^+(\cI_F)[\qis^{-1}] \to \cD^+(\cA).
\]
\end{proposition}

Note that the inverse $\iota_F^{-1}$ replaces a complex $C^\cpx$ by an $F$-adapted resolution.

\begin{definition}
Let $F \colon \cA \to \cB$ be a left-exact functor, and let $\cI_F$ be an $F$-adapted class.
The \emph{right-derived functor} of $F$ is given by
\[
\R F \colon \cD^+(\cA) \to \cD^+(\cB),\ 
C^\cpx \mapsto F( \iota_F^{-1}(C^\cpx) ).
\]
\end{definition}

\begin{definition}
Let $F \colon \cA \to \cB$ be a left-exact functor.
By taking cohomology, we get induced functors 
\[
\R^i F \coloneqq \cH^i(\R F) \colon \cD^+(\cA) \xto{\R F} \cD^+(\cB) \xto{\cH^i} \cB
\]
which are called \emph{$i$-th right-derived functors} of $F$.

Moreover, we can precompose $\R^i F$ with $\cA \to \cD^+(\cA)$ and get induced functors on the abelian level, which we will denote by the same symbol.
\end{definition}

Note that the last step in the definition of $\R^i F$ is taking cohomology of the complex,
in particular, $\R^i F(A) = \R^0 F(A[i])$.
One can check that $\R^0 F$ and $F$ are naturally isomorphic.

\begin{remark}
\label{rem:derived-functor-unique}
The definition of the right-derived functor $\R F$ is based on the choice of an $F$-adapted class.
But one can show that different choices yield isomorphic derived functors,
as the derived functor can be characterised by a universal property.
See \cite[\S III.6.7]{Gelfand-Manin} for more details on this.
\end{remark}

\begin{example}
\label{ex:ext}
Let $\cA$ be an abelian category with enough injectives.
One common way to define $\Ext^i_\cA(C,D)$ is by using an injective resolution $I^\cpx$ of $D$:
\[
\Ext^i_\cA(C,D) \coloneqq \cH^i \Hom_\cA(C,I^\cpx)
\]
Hence we see that $\Ext^i_{\cA}(C,\blank)$ is the $i$-th right-derived functor of $\Hom_\cA(C,\blank)$.
Moreover, we get that 
\[
\Ext^i_\cA(C,D) = \Hom_{\cD(\cA)}(C,D[i]).
\]
\end{example}

\begin{remark}
Let $F \colon \cA \to \cB$ be a right-exact functor.
Since we have already given the definition of a right derived functor, we leave the proper definition of a \emph{left-derived functor} $\L F$ and $i$-th left-derived functor $\L_i F$ as an exercise to the reader.
\end{remark}

The following theorem finally tells us that (right-) derived functors preserve exactness in the derived sense.
As usual, there is also an analoguous statement for left-derived functors.

\begin{theorem}
\label{prop:derivedexact}
Let $F \colon \cA \to \cB$ be a left-exact functor between abelian categories, such that there is an $F$-adapted class in $\cA$.
Then the right-derived functor $\R F \colon \cD^+(\cA) \to \cD^+(\cB)$ is an exact functor.
\end{theorem} 

\begin{proof}
The essential steps are 
\begin{itemize}
\item $\iota_F \colon \Hot^+(\cI_F)[\qis^{-1}] \to \cD^+(\cA)$ is exact and therefore $\iota_F^{-1}$ as well;
\item the functor $\Hot^+(\cI_F)[\qis^{-1}] \to \cD^+(\cB),\ I^\cpx \to F(I^\cpx)$ is also exact, which follows from the adaptedness of $\cI_F$. \qedhere
\end{itemize}
\end{proof}

Note that the combination of \autoref{prop:derivedexact} and \autoref{rem:longexact} gives again long exact sequences, \ie for an exact triangle $C^\cpx \to D^\cpx \to E^\cpx \to C^\cpx[1]$ and a right-derived functor $\R F$, there is the long exact sequence
\[
\cdots \to \R^i F(C^\cpx) \to \R^i F(D^\cpx) \to \R^i F(E^\cpx) \to \R^{i+1} F(C^\cpx) \to \cdots
\]

We close this section with a very important exact functor.

\begin{definition}
Let $\cA$ be a $\kk$-linear category such that $\Hom(A,B)$ is finite-dimensional for any two objects $A,B \in \cA$. An auto-equivalence $S \colon \cA \to \cA$ is called \emph{Serre functor} of $\cA$ if
for all $A,B \in \cA$ there is an isomorphism
\[
\eta_{A,B} \colon \Hom(A,B) \to \Hom(B,S A)^\vee
\]
of $\kk$-vector spaces, which is functorial in both $A$ and $B$.
\end{definition}

We note that without the assumption on the dimension of the $\Hom$-spaces, one runs into problems (as $V^{\vee\vee} \not\cong V$ if $V$ is an infinite-dimensional vector space).

\begin{proposition}
Let $\cA$ be a $\kk$-linear triangulated category with Serre functor $S$.
Then $S$ is an exact functor.
\end{proposition}

\section{The derived category of coherent sheaves}

\noindent
From now on, we will specialise and consider the abelian category $\coh(X)$ of coherent sheaves on a noetherian scheme $X$ over a field $\kk$.
We denote the derived category of $\coh(X)$ by $\Db(X) \coloneqq \cD^b(\coh(X))$.

\subsection{Deriving left-exact functors}
\label{sec:deriveleft}

\AIM{
In this section, we discuss the most prominent left-exact functors in algebraic geometry: $\Hom$, push-forward and sheaf $\SHom$.
}

\begin{theorem}
Let $X$ be a noetherian scheme.
Then there are enough injective objects in the category of quasi-coherent sheaves $\qcoh(X)$.
\end{theorem}

We remark that injective sheaves are hardly finitely generated.
Actually, for our applications this is only a minor technical issue.

\begin{proposition}
Let $X$ be a noetherian scheme. Then the inclusion functor
\[
\Db(X) \to \Db(\qcoh(X))
\]
induces an equivalence of $\Db(X)$ with $\Db_{\coh}(\qcoh(X))$,
the derived category of complexes of quasi-coherent sheaves with bounded cohomology.
\end{proposition}

The following proposition will allow us to restrict derived functors to the bounded derived category of coherent sheaves.

%

\begin{proposition}
\label{prop:restrict-Db}
Let $X$ and $Y$ be schemes and $F \colon \qcoh(X) \to \qcoh(Y)$ be a left-exact functor.
Assume that there is an $F$-adapted class in $\qcoh(X)$.

If for any $\cF \in \coh(X)$ holds $\R F(\cF) \in \Db(Y)$, then the right-derived functor of $F$ restricts to
\[
\R F \colon \Db(X) \to \Db(Y).
\]
\end{proposition}

We have formulated this proposition using an adapted class (even though there are enough injective sheaves), because we will also use the dual statement for right-exact functors.

\begin{proof}
Let $\cF^\cpx \in \Db(X)$ be a bounded complex with $\cF^i = 0$ for $i > n$.
Choose some adapted resolution $\cF^\cpx \to I^\cpx = [\cdots 0 \to I^m \to I^{m+1} \to \cdots]$, which might be in $\cD^+(X)$.
Then its truncation complex
\[
I^{\leq n} \colon \quad [\cdots \to 0 \to I^m \to \cdots \to I^{n-1} \xto{d^n} I^n \to \ker(d^n) \to 0 \to \cdots]
\]
is still quasi-isomorphic to $\cF^\cpx$, but $\ker(d^n) \in \coh(X)$ will not be $F$-acyclic in general.
Nevertheless, by assumption $\R F(\ker(d^n)) \in \Db(Y)$, so we can replace $\ker(d^n)$ by some bounded adapted resolution $J^\cpx$.
One can check that $I^{\leq n}$ and $J^\cpx$ fit together to form a bounded adapted resolution of $F^\cpx$.
\end{proof}

\subsection*{Inner homomorphisms}
Given a quasi-coherent sheaf $\cF \in \qcoh(X)$ on a noetherian scheme $X$ over a field $\kk$,
there is the left-exact functor 
\[
\Hom(\cF,\blank) \colon \qcoh(X) \to \kk\hMod = \qcoh(\Spec \kk).
\]
Since $X$ is noetherian, there are enough injectives in $\qcoh(X)$ and we get
\[
\R\Hom(\cF,\blank) \colon \cD^+(\qcoh(X)) \to \cD^+(\kk\hMod)
\]
Actually, this functor can be extended to complexes $\cF^\cpx \in \Com^-(\qcoh(X))$.

\begin{proposition}
Let $X$ be a smooth and proper variety
and $\cF^\cpx \in \Db(X)$.
Then $\R\Hom(\cF^\cpx,\blank)$ restricts to a functor
\[
\R\Hom(\cF^\cpx,\blank) \colon \Db(X) \to \Db(\kk\hmod).
\]
\end{proposition}

This proposition follows from $\R \Hom(\cF^\cpx,\blank) = \R \Gamma \circ \R \SHom(\cF^\cpx,\blank)$, see \autoref{prop:local-global}, and the corresponding statements for $\R \Gamma$ and $\R \SHom$.

\begin{example}
Let $C$ be a projective curve with a singular point $x$.
Then for the skyscraper sheaf $\kk(x)$ one can show that
$\Ext^i(\kk(x),\kk(x)) \neq 0$ for all $i>0$.
In particular, the image of $\R \Hom(\kk(x),\blank)$ is not contained in $\Db(\kk\hmod)$.
\end{example}

\subsection*{Push-forward}
Let $f \colon X \to Y$ be a morphism of noetherian schemes, which induces the left-exact \emph{push-forward} functor (or \emph{direct image})
\[
f_* \colon \qcoh(X) \to \qcoh(Y).
\]
As $X$ is noetherian, $\qcoh(X)$ has enough injectives, so $f_*$ gives the right-derived functor
\[
\R f_* \colon \cD^+ (\qcoh(X)) \to \cD^+(\qcoh(Y)).
\]
The $\R^i f_* = \cH^i\R f_*$ are also known as \emph{higher direct images} of $f$.

For a noetherian scheme $X$ over a field $\kk$, the push-forward of the structure map $\pi\colon X \to \Spec \kk$ is taking global sections, so $\Gamma = \pi_*$ in this case.
Moreover, for a sheaf $\cF$ we find that its cohomology groups are therefore $H^i(X,\cF) = \R^i \pi_*(\cF)$.

\begin{proposition}
Let $f \colon X \to Y$ be a morphism of noetherian schemes and $\cF \in \qcoh(X)$. Then the derived push-forward restricts to
\[
\R f_*\colon \Db(\qcoh(X)) \to \Db(\qcoh(Y)).
\]
If $f$ is in addition a proper morphism, then $\R f_*$ restricts further to
\[
\R f_*\colon \Db(X) \to \Db(Y).
\]
\end{proposition}

\begin{proof}

Using \autoref{prop:restrict-Db}, the statements can be reduced to the following (deep) theorems:
\begin{itemize}
\item if $\cF \in \qcoh(X)$ then $\R^i f_* \cF$ $X$ are trivial for $i> \dim(X)$;
\item if $f$ is proper and $\cF$ coherent, then all $\R^i f_* \cF$ are coherent.\qedhere
\end{itemize}
\end{proof}

\subsection*{Local homomorphims} 
Let $X$ be a noetherian scheme and $\cF \in \qcoh(X)$.
Then there is the left-exact functor
\[
\SHom(\cF,\blank) \colon \qcoh(X) \to \qcoh(X)
\]
which induces the derived functor
\[
\R\SHom(\cF,\blank) \colon \cD^+(\qcoh(X)) \to \cD^+(\qcoh(X)).
\]
Like in the case of $\Hom$, the sheaf $\cF$ can be replaced by a bounded above complex.

\begin{proposition}
Let $X$ be a smooth and proper variety and $\cF^\cpx \in \Db(X)$.
Then $\R\SHom(\cF^\cpx,\blank)$ restricts to a functor
\[
\R\SHom(\cF^\cpx,\blank) \colon \Db(X) \to \Db(X).
\]
\end{proposition}

\begin{definition}
Let $X$ be a smooth projective variety and $\cF^\cpx \in \Db(X)$.
Then the \emph{dual} of $\cF^\cpx$ is
$\cF^\cpx{}^\vee \coloneqq \R\SHom(\cF^\cpx,\cO_X) \in \Db(X)$.
\end{definition}

\subsection{Deriving right-exact functors}
\label{sec:deriveright}

\AIM{
In this section, we discussion of the most prominent right-exact functors: tensor and pull-back.
}

\begin{remark}
Let $X = \PP^1$ be the projective line over an infinite field. Then there are no (non-zero) projective objects in $\coh(X)$ or $\qcoh(X)$, see \cite[Ex. III.6.2]{Hartshorne}.
\end{remark} 

\subsection*{Tensor product}
This lack of projective objects implies that we still have to work in order to derive tensor product and pull-back.

\begin{theorem}
Let $X$ be a scheme and $\cF$ a quasi-coherent sheaf.
Then the flat sheaves in $\qcoh(X)$ form an adapted class for the left-exact functor $F = \cF \otimes \blank \colon \qcoh(X) \to \qcoh(X)$.
\end{theorem}

\begin{proof}
We check the two properties in the definition of adaptedness.
Let $\cF^\cpx$ be an acyclic complex of quasi-coherent sheaves. Then for a flat sheaf $\cE$, the complex $\cF^\cpx \otimes \cE$ is still acyclic by definition of flatness.

Let $\cG \in \qcoh(X)$.
We use that arbitrary direct sums of flat sheaves are flat and that $\cO_U \in \qcoh(X)$ is flat for any open $U \subset X$.
With this it is easy to build a surjection $\bigoplus_i \cO_{U_i} \onto \cG$ by choosing (local) generators of $\cG$ as a $\cO_X$-module.
\end{proof}

\begin{remark}
If $X$ is a noetherian scheme and $\cF$ a coherent sheaf on $X$,
then $\cF$ is flat if and only if it is locally free.
\end{remark}

In particular, tensoring with a locally free coherent sheaf yields an exact functor.
So with no need to derive, we arrive at the description of the Serre functor in the smooth case.

\begin{theorem}
Let $X$ be a smooth projective variety.
Then the exact functor
\[
\Db(X) \to \Db(X),\
\cF^\cpx \mapsto \cF^\cpx \otimes \omega_X [\dim(X)]
\]
is a Serre functor of $\Db(X)$.
\end{theorem}

Projective objects in abelian categories are characterised by a lifting property dual to the one of \autoref{prop:lifting}.
For an adapted class, like locally free sheaves for the tensor product, such a lifting does not exist in general.

\begin{example}
Consider $X = \PP^1$. The structure sheaf $\cO$ is already locally free,
but tensoring the Euler sequence with $\cO(-2)$ yields another locally free resolution:
\[
\begin{tikzcd}
P^\cpx \ar[d, "f"'] & 0 \ar[r] & \cO(-2) \ar[r] & \cO(-1)^{\oplus 2} \ar[r] \ar[d] & 0 \\
\cO       &        & 0 \ar[r]       & \cO \ar[r]                & 0
\end{tikzcd}
\]
Note that $f$ is a quasi-isomorphism, as $\cH^\cpx(P^\cpx) = \cO$.
The (dual) lifting property of \autoref{prop:lifting} would ask for a map $g$ in the converse direction:
\[
\begin{tikzcd}
\cO \ar[r, two heads, "\id"] \ar[d, dashrightarrow, "g"'] & \cO \ar[d, equal] \\
P^\cpx \ar[r, two heads, "f"] & \cO
\end{tikzcd}
\]
But such a $g$ cannot exist, since $\Hom(\cO,\cO(-1)) = H^0(\PP^1,\cO(-1)) = 0$.
\end{example}

\begin{proposition}
Let $X$ be a scheme and $\cF^\cpx \in \Com^-(\coh(X))$.
Then the right-exact functor $\cF^\cpx \otimes \blank$ induces the left-derived functor
\[
\cF^\cpx \Ltensor \blank \colon 
\cD^-(X) \to \cD^-(X).
\]
If additionally, $X$ is smooth and $\cF^\cpx \in \Db(X)$,
then this functor restricts to
\[
\cF^\cpx \Ltensor \blank \colon 
\Db(X) \to \Db(X).
\]
\end{proposition}

\begin{proof}
The last statement can be shown using the analogue of \autoref{prop:restrict-Db} and the theorem that for smooth varieties any $\cF^\cpx \in \Db(X)$ is quasi-isomorphic to a bounded complex of locally free sheaves of length at most $\dim(X)$.
\end{proof}

\begin{remark}
The $(-i)$-th derived functor of the tensor product is denoted by 
\[
\Tor_i(\cF^\cpx,\blank) \coloneqq \cH^{-i}(\cF^\cpx \Ltensor \blank).
\]
\end{remark}

\subsection*{Pull-back}

Let $f \colon X \to Y$ be a morphism of noetherian schemes.
Note that the \emph{pull-back} (or \emph{inverse image}) $f^*$ is the composition of the exact functor $f^{-1}$ with the tensor product $\cO_X \otimes_{f^{-1} \cO_Y} \blank$.

From this, using flat sheaves as an adapted class, we get the left-derived functor
\[
\L f^* \colon \cD^-(\qcoh(Y)) \to \cD^-(\qcoh(X)) 
\]
which is the composition of $f^{-1}$ and $\cO_X \ltensor_{f^{-1} \cO_Y} \blank$.

\begin{proposition}
Let $f \colon X \to Y$ be a morphism of noetherian schemes with $Y$ smooth.
Then $\L f^*$ restricts to
\[
\L f^* \colon \Db(Y) \to \Db(X).
\]
\end{proposition}

\begin{proof}
As for the tensor product, smoothness of $Y$ implies that we can replace a bounded complex of coherent sheaves by a bounded complex of locally free sheaves.
\end{proof}


\subsection{Compatibilities}
\label{sec:compatibilities}

\AIM{
In this section, we will speak a bit about the interaction between the above introduced derived functors: adjunction of pull-back and push-forward, projection formula and flat base change.
}

\begin{remark}
The crucial technical tool for this section is hidden in \autoref{rem:derived-functor-unique}: the right- (or left-) derived functor associated to a left- (or right-) exact functor is  essentially unique.

In particular, an equality of functors on an adapted class extends to an equality of derived functors.
\end{remark}

As a first application of this remark, we see that the adjunction of $f^*$ and $f_*$ on the abelian level extends to derived categories.

\begin{proposition}
Let $f \colon X \to Y$ be proper morphism of smooth varieties.
Then $\L f^*$ and $\R f_*$ form a pair of adjoint functors, \ie
there is an isomorphism functorial in both arguments
\[
\Hom_{\Db(X)}(\L f^* \cF, \cG) \isom \Hom_{\Db(Y)}(\cF, \R f_* \cG).
\]
\end{proposition}

Another equality of abelian functors is $\Gamma \circ \SHom = \Hom$.

\begin{proposition}
\label{prop:local-global}
Let $X$ be a smooth projective variety and $\cF^\cpx \in \Db(X)$.
Then $\R \Gamma \circ \R\SHom(\cF^\cpx,\blank) = \R \Hom(\cF^\cpx,\blank)$.
\end{proposition}

\begin{proof}
Hidden in this statement is the equality $\R(\Gamma \circ \SHom(\cF^\cpx,\blank)) = \R \Gamma \circ \R \SHom(\cF^\cpx,\blank)$, which follows from the fact that $\SHom(\cF^\cpx,\blank)$ maps injective sheaves to $\Gamma$-acyclic ones.
\end{proof}

The next proposition is the so-called \emph{projection formula}
which, on the abelian level of coherent sheaves, holds for locally free sheaves $\cF$.

\begin{proposition}
Let $f \colon X \to Y$ be proper morphism of smooth varieties and $\cE^\cpx \in \Db(X)$, $\cF^\cpx \in \Db(Y)$.
Then there is a natural isomorphism
\[
\R f_*(\cE^\cpx) \Ltensor \cF^\cpx \isom \R f_*( \cE^\cpx \Ltensor \L f^*(\cF^\cpx)).
\]
\end{proposition}

Finally, there is also the \emph{flat base change}.
For this, note that pull-backs along flat morphisms do not need to be derived. 

\begin{proposition}
Let $u \colon X \to Z$ be a flat morphism and $f \colon Y \to Z$ a proper morphism of smooth varieties.
Consider the fibre product
\[
\begin{tikzcd}
X \underset{\smash{Z}}{\times} Y \ar[r, "v"] \ar[d, "g"'] & Y \ar[d, "f"] \\
X \ar[r, "u"] & Z
\end{tikzcd}
\]
Then for $\cF^\cpx \in \Db(Y)$ there is a natural isomorphism
\[
u^* \R f_*(\cF^\cpx) \isom \R g_* v^*(\cF^\cpx)
\]
and in particular, $u^* \R^i f_*(\cF^\cpx) \cong \R^i g_* v^*(\cF^\cpx)$.
\end{proposition}

\subsection{Fourier-Mukai transforms}
\label{sec:fourier}

\AIM{
We introduce an important class of exact functors between derived categories of coherent sheaves.
}

Throughout this section, let $X$ and $Y$ be smooth projective varieties.
Moreover, we denote the two projections from their product by:
\[
\begin{tikzcd}[column sep=small]
& X \times Y \ar[ld, "q"'] \ar[rd, "p"]\\
X && Y
\end{tikzcd}
\]

\begin{definition}
Let $\cP^\cpx \in \Db(X \times Y)$ then the induced exact functor
\[
\Phi_{\cP^\cpx} \colon \Db(X) \to \Db(Y),\ 
\cE^\cpx \mapsto \R p_*(q^* \cE^\cpx \Ltensor \cP^\cpx)
\]
is called the \emph{Fourier-Mukai transform} with \emph{Fourier-Mukai kernel} $\cP^\cpx$.
\end{definition}

Note that in the above definition, there is no need to derive $q^*$, as a projection is flat.
Moreover, a Fourier-Mukai kernel can be used to define also a Fourier-Mukai transform in the converse direction. To stress the direction, we sometimes write $\Phi^{X\to Y}_{\cP^\cpx}$.

For a Fourier-Mukai kernel $\cP^\cpx \in \Db(X \times Y)$,
its \emph{left and right adjoint Fourier-Mukai kernels} in $\Db(X \times Y)$ are
\[
\cP^\cpx_L \coloneqq \cP^\cpx{}^\vee \otimes p^* \omega_Y[\dim(Y)],
\quad
\cP^\cpx_R \coloneqq \cP^\cpx{}^\vee \otimes q^* \omega_X[\dim(X)].
\]
This notation is justified by the following statement.

\begin{proposition}
\label{prop:fm-adjoints}
Let $\cP^\cpx \in \Db(X \times Y)$, then the Fourier-Mukai transforms
$\Phi_{\cP^\cpx_L}^{Y\to X}$ and $\Phi_{\cP^\cpx_R}^{Y\to X}$ are left and right adjoint to $\Phi_{\cP^\cpx}^{X\to Y}$.
\end{proposition}

Finally, one can ask, whether a given exact functor $F \colon \Db(X) \to \Db(Y)$ is \emph{of Fourier-Mukai type}, \ie can be written as a Fourier-Mukai transform with some kernel.
The central result to this question is the following theorem by Orlov.

\begin{theorem}
\label{thm:orlov}
Let $F \colon \Db(X) \to \Db(Y)$ be an exact fully faithful functor 
and assume that $X$ and $Y$ are smooth projective varieties.
Then there is a $\cP^\cpx \in \Db(X\times Y)$ such that $F \cong \Phi_{\cP^\cpx}$.
Moreover, $\cP^\cpx$ is unique up to isomorphism.
\end{theorem}

\begin{remark}
The original statement also assumes the existence of a left adjoint. Based on work of Bondal and van den Bergh, the existence of both adjoints is automatic. Moreover, over an algebraically closed field of characteristic zero, a non-zero exact full functor $F \colon \Db(X) \to \Db(Y)$ between smooth projective varieties is already faithful.
Details on both can be found in the survey \cite[Prop. 3.5 \& Thm. 3.14]{Canonaco-Stellari}.
\end{remark}

For the following, let $\iota \colon X \isomtext \Delta \subset X\times X$ denote the inclusion of the diagonal $\Delta$, in particular, $\iota_* \cO_X = \cO_{\Delta}$.

\begin{examples}
Let $f \colon X \to Y$ be a morphism between smooth projective varieties
and denote by $\Gamma_f$ its graph in $X \times Y$.
Then push-forward and pull-back are of Fourier-Mukai type:
\[
\R f_* \cong \Phi^{X\to Y}_{\cO_{\Gamma_f}},
\quad
\L f^* \cong \Phi^{Y\to X}_{\cO_{\Gamma_f}}.
\]
Notable special cases are
\begin{itemize}
\item $\id = \Phi_{\cO_\Delta}$, where $\Delta$ is the diagonal;
\item $H^*(X,\blank) = \Phi_{\cO_X}$, using that $X \cong X \times \Spec(\kk)$ and  $H^*(X,\blank) = \R \pi_*$ for $\pi \colon X \to \Spec(\kk)$.
\end{itemize}

\smallskip

The shift functor is of Fourier-Mukai type using the kernel $\cO_\Delta[1]$. 

\smallskip 

The tensor product $\cF^\cpx \ltensor \blank$ is of Fourier-Mukai type, using the kernel $\iota_*(\cF^\cpx)$ with $\iota \colon X \into X \times X$.
By \autoref{prop:fm-adjoints}, also $\SHom(\cF^\cpx,\blank)$ is of Fourier-Mukai type, as it is the left adjoint of the tensor product.
\end{examples}

\begin{proposition}
\label{prop:composing-FM}
The composition of two functors of Fourier-Mukai type is again of Fourier-Mukai type.
\end{proposition}

\begin{example}
Let $X$ be a smooth projective variety. Then the Serre functor $S = \blank \otimes \omega_X [\dim(X)]$ of $\Db(X)$ is of Fourier-Mukai type.
\end{example}

With this \autoref{prop:composing-FM} by Mukai, we can use the above functors as building blocks, yielding a vast array of functors of Fourier-Mukai type.
It is probably fair to say that all geometrically meaningful functors are of Fourier-Mukai type. For further discussion see the survey \cite{Canonaco-Stellari}.

\section{The derived Torelli theorem for K3 surfaces}

\AIM{
This last section serves as a bridge to \cite{Macri-Stellari}:
we discuss the auto-equivalences of the derived category of K3 surfaces.
In this section, the ground field will be $\CC$.
}

For basic facts about K3 surfaces and its Hodge theory needed here, already the recap in \cite[\S 10.1]{Huybrechts} is enough.
We will only recall the global Torelli theorem. For this let $X$ be a K3 surface. 
The standard polarised Hodge structure on the second cohomology $H^2(X,\CC)$, which uses the intersection pairing, can be restricted to the integral cohomology:
\[
H^2(X,\ZZ) = H^{2,0}(X,\ZZ) \oplus H^{1,1}(X,\ZZ) \oplus H^{0,2}(X,\ZZ).
\]
Note that a smooth rational curve $C \subset X$ becomes a $(-2)$-class $[C]$ inside $H^{1,1}(X,\ZZ) \subset H^2(X,\ZZ)$. In particular, the associated reflection 
\[
s_{[C]} \colon H^2(X,\CC) \to H^2(X,\CC),\ 
\alpha \mapsto \alpha + (\alpha,[C]) [C]
\]
is a Hodge isometry, \ie $s_{[C]}$ respects the intersection pairing and the decomposition. 
Moreover, this reflection restricts to an (integral) Hodge isometry
$s_{[C]} \colon H^2(X,\ZZ) \to H^2(X,\ZZ)$.

\begin{theorem}[Torelli]
\label{thm:torelli}
Let $X$ and $Y$ be two K3 surfaces.
Then there is an isomorphism $f \colon X \isomtext Y$ if and only if there exists a Hodge isometry $\phi \colon H^2(X,\ZZ) \isomtext H^2(Y,\ZZ)$.

In this case, there are smooth rational curves $C_1,\ldots,C_m$ on $X$ such that
\[
\phi = \pm s_{[C_1]} \circ \cdots \circ s_{[C_m]} \circ f_*.
\]
\end{theorem}

\subsection*{Derived Torelli}
In the following, we will see that the above statement is the cohomological ``shadow'' of a statement involving the respective derived categories.

Let $f \colon X \to Y$ be a morphism of K3 surfaces.
On the level of rational cohomology, $f$ induces a ring homomorphism
\[
f^* \colon H^*(Y,\QQ) \to H^*(X,\QQ),
\]
the \emph{cohomological pull-back}.
Using Poincar\'e duality, the \emph{cohomological push-forward}
\[
f_* \colon H^*(X,\QQ) \to H^*(Y,\QQ)
\]
can be defined as the dual map to $f^*$.
Given a class $\alpha \in H^*(X\times Y,\QQ)$ the \emph{cohomological Fourier-Mukai transform} with kernel $\alpha$ is
\[
\Phi^H_{\alpha} \colon H^*(X,\QQ) \to H^*(Y,\QQ),\ \beta \mapsto p_*(\alpha.q^*(\beta)).
\]

\begin{definition}
Let $X$ be an algebraic K3 surface.
Then the \emph{Mukai vector} of $E^\cpx \in \Db(X)$ is defined as
\[
v(E^\cpx) \coloneqq (\rk(E^\cpx), c_1(E^\cpx), \rk(E^\cpx) + c_1^2(E^\cpx)/2 - c_2(E^\cpx)).
\]
Moreover, for $\alpha = (\alpha_0,\alpha_1,\alpha_2)$ and $\beta = (\beta_0,\beta_1,\beta_2)$ with $\alpha_k, \beta_k \in H^{2k}(X,\QQ)$, 
the \emph{Mukai pairing} is given as
\[
\pairing{\alpha,\beta} \coloneqq \alpha_1.\beta_1 - \alpha_0.\beta_2 - \alpha_2.\beta_0.
\]
\end{definition}

\begin{remark}
Up to a sign, the Mukai pairing can be seen as a cohomological shadow of the \emph{Euler characteristic} $\chi$. 
To be precise, for $E^\cpx, F^\cpx \in \Db(X)$ holds
\[
- \pairing{v(E^\cpx),v(F^\cpx)} = \chi(E^\cpx,F^\cpx) \coloneqq \sum (-1)^k \dim \R^k \Hom(E^\cpx,F^\cpx).
\]
This follows quite immediate from the Hirzebruch-Riemann-Roch formula.
\end{remark}

The definition of the Mukai pairing is made in such a way, that the pairing extends the intersection pairing on $H^2(X,\ZZ)$.
Even more, we can extend the integral Hodge structure on $H^2(X,\ZZ)$ by setting
\[
\begin{split}
\tilde H^{2,0}(X,\ZZ) &\coloneqq H^{2,0}(X,\ZZ),\\
\tilde H^{1,1}(X,\ZZ) &\coloneqq H^0(X,\ZZ) \oplus H^{1,1}(X,\ZZ) \oplus H^4(X,\ZZ),\\
\tilde H^{0,2}(X,\ZZ) &\coloneqq H^{0,2}(X,\ZZ).
\end{split}
\]
This gives an integral weight-two Hodge structure on $H^*(X,\ZZ)$, which is polarised by the Mukai pairing.
In the following, we will denote this polarised Hodge structure by $\tilde H(X,\ZZ)$, which is called the \emph{Mukai lattice}.

\begin{theorem}[\cite{Mukai}]
\label{thm:mukaiK3}
Let $\Phi \colon \Db(X) \isomtext \Db(Y)$ be an equivalence of the derived categories of two algebraic K3 surfaces.
Then this induces a map on cohomology which defines a Hodge isometry
\[
\Phi^H \colon \tilde H(X,\ZZ) \isom \tilde H(Y,\ZZ).
\]
\end{theorem}

\begin{proof}
By \autoref{thm:orlov}, $\Phi$ can be written uniquely as a Fourier-Mukai transform with Fourier-Mukai kernel $P^\cpx \in \Db(X \times Y)$.
As the key ingredient, Mukai showed that $v(P^\cpx) \in \tilde H^{1,1}(X,\ZZ)$. 
As a consequence, $\Phi^H \coloneqq \Phi^H_{v(P^\cpx)} \colon H^*(X,\QQ) \to H^*(X,\QQ)$ can be restricted to the integral part.
Finally, as an application of the Grothendieck-Riemann-Roch formula, one obtains that
\[
v(\Phi(E^\cpx)) = \Phi^H(v(E^\cpx)). \qedhere
\]
\end{proof}

\begin{remark}
If $\Phi \colon \Db(X) \to \Db(Y)$ is an equivalence between derived categories of arbitrary smooth projective varieties,
then there is a natural pairing on cohomology such that the induced
$\Phi^H\colon H^*(X,\QQ) \to H^*(X,\QQ)$ is an isometry, see \cite[\S 5.2]{Huybrechts}.
Note that in the general situation, $\Phi^H$ will not restrict to the integral part.
\end{remark}

\begin{corollary}
Let $X$ be an algebraic K3 surface. Then there is a homomorphism of groups
\[
\varpi \colon \Aut(\Db(X)) \to \Isom(\tilde H(X,\ZZ)),\ \Phi \mapsto \Phi^H,
\]
where $\Aut(\Db(X))$ is the group of auto-equivalences of $\Db(X)$ and $\Isom(\tilde H(X,\ZZ))$ denotes the group of Hodge isometries of the Mukai lattice.
\end{corollary}

Orlov strengthened the Mukai's result to the so-called derived Torelli Theorem.

\begin{theorem}[\cite{Orlov}]
Two algebraic K3 surfaces $X$ and $Y$ have equivalent derived categories if and only if there exists a Hodge isometry of their Mukai lattices.
\end{theorem}

\begin{proof}
The strategy of the proof is to reduce to the case that an isometry $\phi \colon \tilde H(X,\ZZ) \to \tilde H(Y,\ZZ)$ preserves $H^2$.
This reduction is built on results about moduli spaces of sheaves on K3 surfaces, see \cite[\S 10.3]{Huybrechts} for an overview.
As soon as $\phi$ preserves $H^2$, by \autoref{thm:torelli} such an isometry is of the form
\[
\phi = \pm s_{[C_1]} \circ \cdots \circ s_{[C_m]} \circ f_*.
\]
for some $f \colon X \isomtext Y$. In particular, $\Db(X) \cong \Db(Y)$.
\end{proof}

The relationship of auto-equivalences and Hodge isometries was clarified further by Hosono et.al. \cite{Hosono-etal}, Ploog \cite{Ploog} and Huybrechts, Macr\`i \& Stellari \cite{Huybrechts-Macri-Stellari}.

The central observation is that the Mukai lattice $\tilde H(X,\ZZ)$ has signature $(4,20)$ and that there is a \emph{natural orientation} of the positive directions.
Given an ample class $\alpha \in H^{1,1}(X)$ and a generator $\sigma \in H^{2,0}(X)$, the four classes
\[
\Re(\exp(i \alpha)) = 1-\alpha^2/2,\ 
\Im(\exp(i \alpha)) = \alpha,\ 
\Re(\sigma),\ 
\Im(\sigma)
\]
define an orientation which is independent of the choices of $\alpha$ and $\sigma$.
We denote by $\Isom^+(\tilde H(X,\ZZ))$ the Hodge isometries which preserve this orientation.

\begin{proposition}
\label{prop:isom-plus}
Let $X$ be an algebraic K3 surface and $\Phi \in \Aut(\Db(X))$.
Then $\Phi^H$ preserves the natural orientation, \ie $\Phi^H \in \Isom^+(\tilde H(X,\ZZ)).$
Conversely, for any $\psi \in \Isom(\tilde H(X,\ZZ))$ there is a $\Psi \in \Aut(\Db(X))$ with
\[
\Psi^H = \psi \circ (\pm\id_{H^2}),
\]
in particular, $\Isom^+(\tilde H(X,\ZZ)) \subset \Isom(\tilde H(X,\ZZ)$ has index two.
\end{proposition}

\subsection*{Spherical twists}

As a corollary of \autoref{prop:isom-plus} we obtain the following short exact sequence for an algebraic K3 surface $X$:
\[
0 \to \ker(\varpi) \to \Aut(\Db(X)) \xto{\varpi} \Isom^+(\tilde H(X,\ZZ)) \to 0.
\]
One may wonder which elements lie in the kernel of $\varpi$, \ie auto-equivalences that act as the identity on cohomology.
Or one might ask whether the reflections $s_{[C]}$ for smooth rational curves $C \subset X$ can be lifted to auto-equivalences of $\Db(X)$.
Both questions lead to the notion of a spherical twist.
We recall the central properties, for further details see \cite[\S 8.1]{Huybrechts}.

\begin{definition}
Let $X$ be a smooth projective variety.
An object $E^\cpx \in \Db(X)$ is called \emph{spherical} if
\begin{itemize}
\item $E^\cpx$ is a \emph{spherelike} object, \ie 
\[
\Hom^*(E^\cpx,E^\cpx) \coloneqq \bigoplus_k \Hom(E^\cpx,E^\cpx[k])[-k] \cong \CC[t]/t^2;
\]
\item and $E^\cpx$ is a \emph{Calabi-Yau} object, \ie $E^\cpx \otimes \omega_X \cong E^\cpx$.
\end{itemize}
\end{definition}

\begin{remark}
The graded vector space $\Hom^*(E^\cpx,E^\cpx)$ is the cohomology of the complex $\R\Hom(E^\cpx,E^\cpx) \in \Db(\CC\hmod)$, 
and actually quasi-isomorphic to it. 
Note that $\Hom^*(E^\cpx,E^\cpx)$ becomes a $\CC$-algebra with the Yoneda product.
So the first property asks that there is an (up to scalar) unique self-extension of $E^\cpx$ that squares to zero. 
By the second property this extension has to be of degree $\dim(X)$.
\end{remark}

\begin{theorem}
Let $E^\cpx$ be a spherical object in $\Db(X)$.
Then there is an auto-equivalence $\T_{E^\cpx}$ of $\Db(X)$ which fits into an exact triangle of functors:
\[
\Hom^*(E^\cpx,\blank) \otimes E^\cpx \xto{\mathsf{ev}} \id \to \T_{E^\cpx} \to \Hom^*(E^\cpx,\blank) \otimes E^\cpx [1],
\]
which is called the \emph{spherical twist} along $E^\cpx$.
\end{theorem}

\begin{remark}
The first arrow in the above triangle is the \emph{evaluation map}  $\mathsf{ev}$ which comes from the adjunction of $\Hom$ and $\otimes$.
This triangle of functors cannot serve as a definition of $\T_{E^\cpx}$, as cones are not functorial. But $\mathsf{ev}$ induces a morphism between the respective Fourier-Mukai kernels, which allows to define $\T_{E^\cpx}$ as a Fourier-Mukai transform to the cone of this morphism.
\end{remark}

The triangle of functors above allows to deduce easily two important properties of a spherical twist $\T_{E^\cpx}$:
\begin{itemize}
\item $\T_{E^\cpx}(E^\cpx) = E^\cpx[1-\dim(X)]$ and
\item $\T_{E^\cpx}(F^\cpx) = F^\cpx$ for $F^\cpx$ with $\Hom^*(E^\cpx,F^\cpx)=0$.
\end{itemize}
On the level of cohomology, these properties become
\begin{itemize}
\item $\T_{E^\cpx}^H(v(E^\cpx)) = (-1)^{1-\dim(X)}v(E^\cpx)$ and
\item $\T_{E^\cpx}^H(\alpha) = \alpha$ for $\alpha$ with $\pairing{v(E^\cpx),\alpha}=0$.
\end{itemize}
In particular, $\T_{E^\cpx}^H$ is already completely determined: it is the reflection along $v(E^\cpx)^\perp$ if $\dim(X)$ is even and the identity if $\dim(X)$ is odd.

\begin{corollary}
For an algebraic K3 surface $X$ holds:
\begin{itemize}
\item If $E^\cpx$ is a spherical object, then $\T_{E^\cpx}^2$ is a non-trivial element of $\ker(\varpi)$.
\item If $C \subset X$ is a smooth rational curve, then $\cO_C(-1)$ is a spherical object with $\T_{\cO_C(-1)}^H = s_{[C]}$.
\end{itemize}
\end{corollary}

\begin{proof}
The first part follows from the observation that $\T_{E^\cpx}^H$ is a reflection. 
For the second part, one can check that all $\cO_C(k)$ are spherical objects for $k \in \ZZ$. But only for $k=-1$, one obtains that $v(\cO_C(-1)) = [C]$.
\end{proof}

So in the presence of smooth rational curves on an algebraic K3 surface, we obtain elements in $\ker(\varpi)$. 
The question about the structure of $\ker(\varpi)$ in general is hard, so far only the case of Picard rank $1$ is solved.

\begin{theorem}[\cite{Bayer-Bridgeland}]
Let $X$ be an algebraic K3 surface of Picard rank $1$.
Then $\ker(\varpi)$ is the product of $\ZZ \cdot [2]$ and the free group generated by $\T_V^2$ with $V$ running over all spherical vector bundles on $X$.
\end{theorem}

\newcommand{\bib}[5]{{\bibitem[#1]{#2} #3, {\it #4}, #5.}}
\newcommand{\arXiv}[1]{\href{http://arxiv.org/abs/#1}{\texttt{arXiv:#1}}}

\end{document}